\documentclass[12pt]{article}
\usepackage{graphicx,amsmath,amssymb,amsthm,rotating,stmaryrd,datetime,pifont,booktabs,enumitem}
\usepackage[mathscr]{eucal}

\usepackage{setspace}
\usepackage{enumitem}
\setlist{nolistsep}
\setlist{nosep}

\usepackage[usenames,dvipsnames]{color}
\newcommand\red[1]{{\color{red} #1}}

\newtheorem{theorem}{Theorem}

\newtheorem{result}[theorem]{Result}
\newtheorem{lemma}[theorem]{Lemma}

\newtheorem{construction}[theorem]{Construction}
\newtheorem{conjecture}[theorem]{Conjecture}

\newcommand{\new}{\nu}
\newcommand{\f}{f}
\newcommand{\g}{g}

\renewcommand{\O}{\mathcal O}
\newcommand{\A}{\mathcal A}

\renewcommand{\H}{\mathcal{H}}
\renewcommand{\P}{\mathcal{P}}
\renewcommand{\S}{\mathcal{S}}

\newcommand{\R}{\mathcal{R}}
\newcommand{\C}{\mathcal{C}}
\newcommand{\D}{\mathcal{D}}

\newcommand{\li}{\ell_\infty}
\newcommand{\si}{\Sigma_\infty}

\newcommand{\B}{{\mathbf{B}}}

\newcommand{\F}{\mathcal{F}}
\newcommand{\PG}{{\textup{PG}}}
\newcommand{\AG}{{\textup{AG}}}

\newcommand{\Fqt}{\mathbb{F}_{\hspace*{-1mm}{q^{t}}}}
\newcommand{\Fq}{\mathbb{F}_{\hspace*{-.5mm}q}}

\newcommand{\Fqts}{\mathbb{F}^*_{\hspace*{-0.5mm}{q^{t}}}}
\newcommand{\hyp}{\frac{q^t-1}{q-1}}

\DeclareMathOperator{\Nm}{N}

\begin{document}

\title{Inherited arcs in Andr\'e planes of even order
}

\author{S.G. Barwick, Alice M.W. Hui and Wen-Ai Jackson}
\date{}
\maketitle




\begin{abstract} Andr\'e planes of order $q^t$ can be constructed from the Desarguesian plane  $\PG(2,q^t)$ using a process that involves replacing Andr\'e nets. This article constructs ovals in  Andr\'e planes of order $q^t$, $q$ even and $t$ prime that are inherited from conics of the  Desarguesian plane  $\PG(2,q^t)$. Further, it is shown that certain conics do not inherit to arcs under the process of Andr\'e replacement.
\end{abstract}

\section{Introduction}

The process of deriving the Desarguesian plane  $\PG(2,q^2)$  replaces certain lines with Baer subplanes, and constructs the Hall plane. A well studied problem involves taking  a conic in $\PG(2,q^2)$ and using the process of derivation to construct inherited arcs in the Hall plane. The effect of derivation on conics of $\PG(2,q^2)$ has been completely determined, with contributions by a number of authors, see \cite{barmar}, \cite{BKNS}, \cite{Cher}, \cite{GlynnStein}, \cite{Korch1}, \cite{Korch2}, \cite{OKeefePasPen}, \cite{OKeefePas}, \cite{Szon}.
In this article we are interested in the case when $q$ is even, and the complete results when $q$ is even are as follows.

\begin{result}\label{120e}
Let $\C$ be a non-degenerate conic in $\PG(2,q^2)$, $q$ even, $q>2$, and $\D$ a Baer subline of $\li$. Derive $\PG(2,q^2)$ using $\D$ to get the Hall plane $\H(\D)$. The affine points of $\C$ correspond to a set of affine points of $\H(\D)$ denoted by $\D(\C)$.
\begin{enumerate}
\item \cite{Cher,OKeefePas} If $\C$ is secant to $\li$, then $\D(\C)$ is not an arc of the Hall plane.
\item \cite{Cher} If $\C$ is exterior to $\li$, then $\D(\C)$ is not an arc of the Hall plane.
\item Suppose $\C$ is tangent to $\li$ at the point $T$ and $\C$ has nucleus $N$.
\begin{enumerate}
\item \cite{OKeefePas} If $T,N\in\D$, then $\D(\C)$ is not an arc of the Hall plane.
\item \cite{OKeefePasPen} If exactly one of $T,N$ lies in $\D$ then $\D(\C)$ is an arc of the Hall plane.
\item \cite{GlynnStein} If $T,N\notin\D$, then $\D(\C)$ is an arc of the Hall plane if and only if $q$ is a square and $T,N$ are
 conjugate with respect to $\D$.
\end{enumerate}
\end{enumerate}
\end{result}

The process of deriving $\PG(2,q^2)$ using a Baer subline of $
\li$  to get the Hall plane can be generalised to performing an Andr\'e replacement of $\PG(2,q^t)$, $t\geq3$, using an Andr\'e set of $\li$ to get an Andr\'e plane of order $q^t$.
This process of Andr\'e replacement is described in Section \ref{sec2}. Note that given a Baer subline of $\li$ in $\PG(2,q^2)$, there is a unique way to derive $\PG(2,q^2)$ to get the Hall plane. In $\PG(2,q^t)$, $t\geq3$,  given an Andr\'e set of $\li$, there are $t-1$ different Andr\'e replacement nets, leading to $t-1$ different ways to perform an Andr\'e replacement and construct an associated Andr\'e plane. This means that when studying whether conics inherit to arcs when $t\geq3$, we need to consider all of the $t-1$  associated Andr\'e planes.

The main results of  this article are   Theorems~\ref{thm000},  \ref{201}, and  \ref{014}  which  generalise Result~\ref{120e} to partially answer when conics in $\PG(2,q^t)$, $q$ even, $t\geq3$ prime, are inherited to arcs in Andr\'e planes.
In particular,  Theorem~\ref{thm000} generalises  parts 1 and 2 of Result \ref{120e} to show that conics of $\PG(2,q^t)$, $q$ even, $t$ a prime, that are secant or exterior to $\li$ never inherit to arcs of an Andr\'e plane. When $\C$ is tangent to $\li$ there are several cases to consider. Theorem  \ref{201} generalises part 3a) of Result \ref{120e}  to $\PG(2,q^t)$, $q$ even, $t$ a prime, and we never get an inherited arc in this case. Theorem \ref{014} generalises part 3c)  of Result \ref{120e} to give an example of an inherited arc -- the result for this case is more complex when $t\geq3$  than    the $t=2$ case, and the generalisation is not straightforward.
This leaves several   open cases  which are discussed in Section~\ref{sec-con}, in particular the case of an inherited arc given in Result \ref{120e}(3b) when $t=2$ is conjectured to not give an inherited arc when $t\geq3$.

\section{Background}\label{sec2}

The finite field of prime power order $q$ is denoted $\Fq$, and we let $\Fq^*=\Fq\setminus\{0\}$. The \emph{norm mapping} from $\Fqt$ to $\Fq$ is denoted $\Nm (k)=k^{q^{t-1}+\cdots+q+1}$  for $k\in\Fqt$.

The Bruck-Bose representation \cite{andr54,bruc69,bruc64,segre} of $\PG(2,q^t)$ in $\PG(2t,q)$  is well known. Let $\si$ be the hyperplane at infinity in $\PG(2t,q)$ and let $\S$ be a   $(t-1)$-spread in $\si$. The incidence structure $\A(\S)$ with points the points of $\PG(2t,q)\setminus\si$; lines the $t$-spaces of $\PG(2t,q)\setminus\si$ that contain an element of $\S$; and incidence being inclusion, is an affine plane. We can complete this to a projective plane denoted $\P(\S)$: points on the line at infinity $\li$ correspond to the elements of $\S$. Moreover, $\P(\S)\cong \PG(2,q^t)$ iff $\S$ is a regular (Desarguesian) spread.
If $\mathcal X$ is a set of points of  $\PG(2,q^t)$, then we denote the corresponding set of points in $\PG(2t,q)$ by $[\mathcal X]$.

The representation of $\li\cong\PG(1,q^t)$ as a regular $(t-1)$-spread in $\si\cong\PG(2t-1,q)$ is often referred to as  field reduction, and denoted by $F_{2,t,q}$. Coordinates for this will be denoted as follows.  If $P=(x,y)\in\PG(1,q^t)$, then the corresponding $(t-1)$-space in $\PG(2t-1,q)$ is denoted $[P]=F_{2,t,q}(P)$ and has coordinates written as $[P]=\{\langle (sx,sy)\rangle_q\,|\,s\in\Fqt^*\}$.

   An Andr\'e set of $\PG(1,q^t)$, $t\geq3$, can be  constructed  as follows.
Let $E,F$ be two distinct points of $\PG(1,q^t)$ and let $G$ be the collineation group acting on points of $\PG(1,q^t)$ that fixes $E$ and $F$ pointwise. As $G$ is Singer group of order $q^t-1$, $G$ has a unique Singer subgroup $H$ of order $q^{t-1}+\cdots+q+1$. The orbits of $H$ partition the  points of $\PG(1,q^t)\setminus\{E,F\}$ into $q-1$ sets  of size $q^{t-1}+\cdots+q+1$, 
these sets are called \emph{Andr\'e sets} and $E,F$ are called 
 the \emph{transversal points} of the Andr\'e set.
 Andr\'e sets are defined in multiple settings, and a short discussion of the different settings and different nomenclature is given in \cite{qoddpaper}. In particular, in the linear set literature, Andr\'e sets are scattered $\Fq$-linear sets of $\PG(1,q^t)$ of pseudoregulus type. 
We note that an Andr\'e set of $\PG(1,q^t)$, $t\geq3$, is projectively equivalent to $$
\A_1=\{(1,\theta)\,|\,k\in\Fqt^*,\Nm(\theta)=1\}$$ which has transversal points $(1,0)$, $(0,1)$.
 Moreover the $q-1$ Andr\'e sets with transversal points $(1,0)$, $(0,1)$ are the sets $\A_\delta$, $\delta\in\Fq^*$ with  coordinates given by
$$
\A_\delta=\{(1,\theta)\,|\,k\in\Fqt^*,\Nm (\theta)=\delta\}.
$$

Under field reduction, the Andr\'e set $
\A_\delta=\{(1,\theta)\,|\,k\in\Fqt^*,\Nm (\theta)=\delta\}
$ of $\PG(1,q^t)$ corresponds to a set  $[\A_\delta]$ of $\hyp$ elements of the regular (Desarguesian) $(t-1)$-spread $\S$ in $\PG(2t-1,q)$.
This set is  called  an \emph{Andr\'e net} and has coordinates:
$$ [\A_\delta] =\{[P_k]=\{\langle(s,sk)\rangle_q\,|\, s\in\Fqt^*\}\ \,|\,k\in\Fqt^*, \Nm(k)=\delta\}.$$
In $\PG(2t-1,q)$, there are $t-1$ \emph{Andr\'e replacement nets} for $[\A_\delta]$, denoted $\A_\delta^m$, $m=1,\ldots,{t-1}$. Coordinates for these nets are given in multiple places, see \cite{qoddpaper} for a short survey and for a discussion on why we only consider  the case $t$ prime.
The coordinates for the Andr\'e replacement nets are given by the following:
\begin{equation}
\label{e2}
\A_\delta^m=\{X_{m,u}=\{\langle(s,s^{q^m}u)\rangle_q\,|\, s\in\Fqt^*\}\ \,|\,u\in\Fqt^*, \Nm(u)=\delta\}.
\end{equation}

We can form a new spread of $\PG(2t-1,q)$ from the regular $(t-1)$-spread $\S$, namely the  spread $(\S\setminus [\A_\delta])\cup\A_\delta^m$ for some $m\in\{1,\ldots,t-1\}$.
Recall that as $\S$ is regular, the Bruck-Bose plane $\P(\S)\cong \PG(2,q^t)$ is Desarguesian. The plane $\P\big((\S\setminus [\A_\delta])\cup\A_\delta^m\big)$ is   an Andr\'e plane,  we abbreviate the notation and write  $\P(\A_\delta^m)=\P\big((\S\setminus [\A_\delta])\cup\A_\delta^m\big)$. The process of constructing $\P(\A_\delta^m)$ from $\P(\S)\cong \PG(2,q^t)$ is called \emph{Andr\'e replacement}, and the next statement carefully describes it (using a generic notation for the Andr\'e set).

\begin{construction}\label{cons1} {\bfseries \emph{(Andr\'e replacement)}} In $\PG(2,q^t)$, $t$ prime, $t\geq 3$, let $\D$ be an Andr\'e set of $\li$. In the $\PG(2t,q)$ Bruck-Bose setting, $[\D]$ is an Andr\'e net of the regular spread $\S\subset\si$,   denote the $t-1$ Andr\'e replacement nets by $\D^m$, $m=1,\ldots,t-1$.
 For $m\in\{1,\ldots,t-1\}$, define  the incidence structure $\mathcal T_m$ as follows.
\begin{itemize}
\item The points of $\mathcal T_m$ are the
 points of $\AG(2,q^t)$
 \item The lines of $\mathcal T_m$ consist of
 \begin{itemize}
\item  the lines of $\AG(2,q^t)$ that meet $\li$ in a point not in $\D$,
\item the \emph{$\D^m$-blocks} (a set of affine points $\mathcal X$ in $\PG(2,q^t)$ is called an \emph{$\D^m$-block} if in $\PG(2t,q)$, the corresponding set $[\mathcal X]$ is an affine $t$-space whose projective completion  meets $\si$ in a $(t-1)$-space that lies in the replacement net $\D^m$).
\end{itemize}
\item Incidence is inclusion
\end{itemize}
Then $\mathcal T_m$ is an affine plane of order $q^t$ whose projective completion is an Andr\'e plane denoted $\P(\D^m)$. This process is called an \emph{Andr\'e replacement with respect to the Andr\'e replacement net $\D^m$. }
\end{construction}

 Recall that a $(t-1)$-regulus in $\PG(2t-1,q)$ is a set $\R$ of $q+1$ pairwise disjoint $(t-1)$-spaces, with the property that if a line meets three elements of $\R$, then it meets all elements of $\R$. Hence a  $(t-1)$-regulus is  ruled by a set of $q^{t-1}+\cdots+q+1$ (pairwise disjoint) lines.
It is well known that if $b$ is an $\Fq$-line contained in $\li$ in $\PG(2,q^t)$, then in the $\PG(2t,q)$ Bruck-Bose representation, $[b]$ is a $(t-1)$-regulus contained in the regular spread $\S$. Conversely, every $(t-1)$-regulus contained in $\S$ corresponds to an $\Fq$-line of $\li$.

When $t$ is prime, the $\Fq$-lines contained in an Andr\'e set  $\D$ can be partitioned into $t-1$ families. We need the following relationship between the $t-1$ families of $\Fq$-lines and the $t-1$ Andr\'e replacement nets for $\D$; this key result motivates why we only consider the case when $t$ is prime.

\begin{result}\cite{qoddpaper} \label{M110}
Let $\D$ be an Andr\'e set of $\PG(1,q^t)$, $t\geq3$ prime, $q\geq t$.
\begin{enumerate}
\item\label{M110a} The $\Fq$-lines contained in $\D$ can be partitioned into $t-1$ families of size $(q^{t-1}+\ldots+q+1)(q^{t-1}-1)/(q^2-1)$. We can label the families by $\F_1,\ldots,\F_{t-1}$ so that for $i\in\{1,\ldots,t-1\}$, family $\F_i$ corresponds to
the  Andr\'e replacement net $\D^i$  in the  following way.
\begin{enumerate}
\item If $\ell$ is a line contained in one of  the $(t-1)$-spaces in  $\D^i$, then the set $\{P\in\li\,|\,[P]\cap\ell\neq\emptyset\}$  is an $\Fq$-line belonging to family $\F_i$.
\item  If
$b$ is an $\Fq$-line in family $\F_i$, then  every ruling line of the $(t-1)$-regulus $[b]$ is contained in a unique
 $(t-1)$-space of $\D^i$.
 \end{enumerate}

\item\label{M110b}  Any two distinct points of $\D$ lie in exactly $t-1$ $\Fq$-lines of $\D$, one from each family.
 \end{enumerate}
  \end{result}

These families of $\Fq$-lines are critical to studying inherited arcs due to the following relationship with the   $\D^m$-blocks defined in Construction~\ref{cons1}.

\begin{result}\cite{qoddpaper}\label{001}
In $\PG(2,q^t)$,  $q\geq t$, $t\geq3$ prime, let $\D$ be an Andr\'e set of $\li$, and denote the families of $\Fq$-lines contained in $\D$ by $\F_1,\ldots,\F_{t-1}$.  Three non-collinear affine points $A,B,C$ lie in a common  $\D^m$-block for some $m\in\{1,\ldots,t-1\}$ if and only if the $\Fq$-line determined by the three points $AB\cap\li$, $AC\cap\li$, $BC\cap\li$  belongs to family $\F_m$.
\end{result}



\section{Secant and exterior conics}

In this section we look at a conic which is either secant or exterior to $\li$ in $\PG(2,q^t)$, $q$ even, and show that an Andr\'e replacement never results in an inherited arc.

The article
\cite{BKNS} showed how powerful the following result is in the case when $t=2$, that is, when deriving $\PG(2,q^2)$ and looking at which conics inherit to arcs of the Hall plane. It is also powerful when $t\geq3$ and can be used to look at the cases when a conic is secant or  exterior to $\li$.

\begin{result}\label{result000} \cite{segrekorch}
In $\PG(2,q)$, $q$ even, let $\C$ be a conic and $\ell$  a line that  is not tangent to $\C$. For any triple $\{P_1 , P_2 , P_3 \}$ of points contained in $\ell\setminus\C$, there is exactly one triangle $\{A_1,A_2,A_3\}$ inscribed in $\C \setminus \ell$ such that $A_i A_j \cap \ell = P_k$, where $i, j, k$ is a permutation of $1, 2, 3$.
\end{result}

\begin{theorem}\label{thm000}  In $\PG(2,q^t)$, $q$ even, $q\geq t$, $t\geq3$ prime, let $\C$ be a conic and $\D$ be an Andr\'e set of $\li$. Perform an Andr\'e replacement using the Andr\'e replacement net $\D^m$, and let $\D_m(\C)$ denote the set of points in $\P(\D^m)$ corresponding to the affine points of $\C$.    If   $\C$ is a conic that is either exterior or secant to $\li$, then $\D_m(\C)$ is not an arc in the Andr\'e plane $\P(\D^m)$ for any $m\in\{1,\dots,t-1\}$.
\end{theorem}

\begin{proof} Let $\C$ be a conic that is exterior of secant to $\li$.
Let $m\in\{1,\ldots,t-1\}$ and let $b$ be an $\Fq$-line in  $\D$ belonging to family $\F_m$. Let $P_1,P_2,P_3$ be three points in $b$ which do not lie in $\C$. By Result \ref{result000}, there are points $A_1,A_2,A_3 \in \C$ such that $A_iA_j \cap l_\infty = P_j$ for any $\{i,j,k\}=\{1,2,3\}$. By Result~\ref{001}, the points $A_1,A_2,A_3$ lie in a $\D^m$-block. This $\D^m$-block corresponds to a line in $\P(\D^m)$ that meets $\D_m(\C)$ in three points, so $\D_m(\C)$ is not an arc.
\end{proof}

\section{Tangent conics}

In this section we look at a conic  in $\PG(2,q^t)$ which is tangent to $\li$ and try to determine whether an Andr\'e replacement leads to an inherited arc.

Let $\C$ be a conic  in $\PG(2,q^t)$  tangent  to  $\li$ at the point $T$ with nucleus $N$, and let
$\D$ be  an Andr\'e set with transversal points $E,F$. There are a number of cases to consider, depending on the intersection of $\C\cup N$ and the set $\D\cup \{E,F\}$.  In this section we consider the following two cases.

Firstly, Section \ref{tgt-sec1} looks at the case when both $T,N\in\D$ and shows that in this case we never get an inherited arc -- this directly generalises Result \ref{120e}(3a).

Section \ref{tgt-sec2} looks at the case when $\{T,N\}=\{E,F\}$, (that is, where $\{T,N\}$ are the transversal points of $\D$) and shows that in most cases we get an inherited arc. The main result is stated in Theorem~\ref{014} -- this result generalises Result \ref{120e}(3c), although the generalisation is not an obvious one.

\subsection{Tangent conics - the case where  $T,N\in\D$}\label{tgt-sec1}

\begin{theorem}\label{201} In $\PG(2,q^t)$, $q$ even, $q\geq t$, $t\geq3$ prime, let $\C$ be a conic and $\D$ be an Andr\'e set of $\li$. Perform an Andr\'e replacement using the Andr\'e replacement net $\D^m$, and let $\D_m(\C)$ denote the set of points in $\P(\D^m)$ corresponding to the affine points of $\C$.   If  $\C$ is a conic with nucleus $N$ and tangent to the line $\li$ at the point $T$,   with $T,N\in\D$, then $\D_m(\C)$ is not an arc in the Andr\'e plane $\P(\D^m)$ for any $m\in\{1,\dots,t-1\}$.
\end{theorem}

\begin{proof}
Let $m\in\{1,\ldots,t-1\}$.
By Result~\ref{M110},
there is a unique $\Fq$-line $b$ contained in family $\F_m$ that contains the two points $T,N$. Let $X$ be a point of $b$ distinct from $T,N$ and let $\ell$ be a $2$-secant of $\C$ through the point $X$, with $\ell\cap\C=\{A,B\}$.
The term  $\Fq$-conic contained in $\C$ refers to a set of $q+1$ points of $\C$ that are projectively equivalent to a non-degenerate conic contained in $\PG(2,q)$. As $\C$ is projectively equivalent to the set $\{(\theta^2,\theta,1)\,|\,\theta\in\Fqt\cup\{\infty\}\}$, it follows that any three points of $\C$ are contained in a unique $\Fq$-conic contained in $\C$.
Let $\O$ be the unique $\Fq$-conic contained in $\C$ that contains the
 three points $T,A,B$.
As $q\geq4$, $\O$ lies in a unique  $\Fq$-plane which we denote by $\alpha$.
As $T\in\O$, $\O$ has nucleus $N$, so $N\in\alpha$. Hence $\alpha$ is secant to $\li$ and so $\alpha\cap\li=b$. As $q\geq4$, $\O$ contains a point $C$ distinct from $T,A,B$. The three points $AB\cap\li$, $AC\cap\li$, $BC\cap\li$ lie in $b$. It follows from Result~\ref{001} that $A,B,C$ lie in a common $\D^m$-block, and so they are collinear in $\P(\D^m)$. Hence $\D_m(\C)$ is not an arc in  $\P(\D^m)$.
\end{proof}

 \subsection{Tangent conic -  the case where   $\{T,N\}$ are the transversal points of $\D$}\label{tgt-sec2}

\subsubsection{Preliminaries}

In what follows,  let $\tau$ be a multiplicative generator for $\Fqt^*$.
Recall from Section~\ref{sec2} that the  $q-1$ Andr\'e sets
with transversal points $(1,0),(0,1)$ are projectively equivalent to
$\A_\delta=\{(1,\theta)\,|\,\theta\in\Fqts,\Nm(\theta)=\delta\}$, $\delta\in\Fq^*$. We rewrite this in a different form.
If $\theta\in\Fqts$, with $ \Nm(\theta)=\delta$, then $\theta=\new \tau^{(q-1)i}$ for some $\nu\in\Fqts$ with $\Nm(\nu)=\delta$ and  $i\in I$ where
 $$I=\{0,1,\ldots,q^{t-1}+\cdots+q\}$$ is an index set.
So the $q-1$ Andr\'e sets
with transversal points $(1,0),(0,1)$ are projectively equivalent to sets $\A_\new$, $\new\in\Fqt^*$ where
$\A_{\new}=\{(1,\new\tau^{(q-1)i})\,|\,i\in  I   \}.$

For the remainder of this section, we let $\D$ be  an Andr\'e set with transversal points $T,N$, and write
\begin{equation}\label{e5}\D=\A_{\new}=\{(1,\new\tau^{(q-1)i})\,|\,i\in  I   \} \quad\textup{for some }\new\in\Fqt^*.\end{equation}

The $t-1$ Andr\'e replacement nets for $\D=\A_\new$ are denoted  $\D^m$, $m=1,\ldots, t-1$; the coordinates can be calculated using (\ref{e2}), giving
\begin{equation}\label{e6}
\D^m=\{ X_{m,u}=\{\langle(s,s^{q^m}u)\rangle_q\,|\, s\in\Fqt^*\}\ \,|\,u\in\Fqt^*, \Nm(u)=\Nm(\new)\}.
\end{equation}

%
%

\begin{lemma}\label{002}
In $\PG(2,q^t)$, $q$ even, $t\geq3$ prime, let $\C$ be a conic with nucleus $N$ and tangent to the line $\li$ at the point $T$ with $\{T,N\}=\{(1,0,0),(0,1,0)\}$. Let
$\D$ be  an Andr\'e set with transversal points $T,N$.
\begin{enumerate}
\item If there exists three     affine points $A,B,C$ in $\C$ with $AB\cap\li$, $AC\cap\li$, $BC\cap\li$  in  $\D$, that is   $AB\cap \li=(1,\new \tau^{(q-1)i},0)$,  $AC\cap \li=(1,\new \tau^{(q-1)j},0)$, $BC\cap \li=(1,\new \tau^{(q-1)k},0)$ for distinct $i,j,k\in I $, then
$
\tau^{(q-1)i}+\tau^{(q-1)j}+\tau^{(q-1)k}=0$.
\item  If  there exists distinct $i,j,k\in I $ satisfying $
\tau^{(q-1)i}+\tau^{(q-1)j}+\tau^{(q-1)k}=0$, then there exists $A,B,C\in\C$ with
 $AB\cap \li=(1,\new \tau^{(q-1)i},0)$,  $AC\cap \li=(1,\new \tau^{(q-1)j},0)$, $BC\cap \li=(1,\new \tau^{(q-1)k},0)$ which lie in $\D$.
\end{enumerate}
\end{lemma}

\begin{proof} First suppose that $T=(0,1,0)$ and $N=(1,0,0)$, so $\C$ has equation $\f x^2+\g z^2+yz=0$ for some $\f,\g\in\Fqt$, $\f\neq0$.
By (\ref{e5}), if $\D$ is an Andr\'e set with transversal points $(1,0,0),(0,1,0)$, then
$\D=\{(1,\new\tau^{(q-1)i},0)\,|\,i\in  I   \}$ for some  $\new\in\Fqts$.

For part 1, suppose there exists three     affine points $A,B,C$ in $\C$ such that the three points $AB\cap\li$, $AC\cap\li$, $BC\cap\li$ lie in  $\D$. The points $A,B,C$ have coordinates $A=(a, a^2 \f+\g,1)$, $B=(b, b^2 \f +\g,1)$, $C=(c, c^2 \f +\g,1)$ for some distinct  $a,b,c\in\Fqt$. We compute the following coordinates:
$P_1=AB\cap\li=(1,(a+b) \f,0)$,
$P_2=AC\cap\li=(1,(a+c) \f,0)$, 
$P_3=BC\cap\li=(1,(b+c) \f,0)$.
So if the three distinct points $P_1,P_2,P_3$ lie in $\D$, then there exists distinct $i,j,k\in I $ with $(a+b)\f=\new \tau^{(q-1)i}$,
$(a+c)\f= \new\tau^{(q-1)j}$,
$(b+c)\f= \new \tau^{(q-1)k}$. As $q$ is even, summing these three equations gives  $\new\tau^{(q-1)i}+\new\tau^{(q-1)j}+\new\tau^{(q-1)k}=0$ and so $\tau^{(q-1)i}+\tau^{(q-1)j}+\tau^{(q-1)k}=0$ as required.

 For part 2, suppose there exists distinct $i,j,k\in I $ satisfying $\tau^{(q-1)i}+\tau^{(q-1)j}+\tau^{(q-1)k}=0.$ Let $c\in\Fqt$ and let
\begin{equation}\label{e07}
a=\new\tau^{(q-1)j} \f^{-1}-c,\quad 
b=\new\tau^{(q-1)k} f^{-1} -c. 
\end{equation}
 Then $a,b,c$ are distinct, so the  points with coordinates $C=(c, c^2 \f+\g,1)$, $A=(a, a^2 \f +g,1)$, $B=(b, b^2 \f +g,1)$ are  distinct affine points of $\C$. Computing  coordinates and using   (\ref{e07}) we have   $AC\cap\li=(1,(a+c)\f,0)=(1,\new\tau^{(q-1)j},0)$ and $BC\cap\li=(1,(b+c)\f,0)=(1,\new\tau^{(q-1)k},0)$. Further $AB\cap\li=(1,(a+b)\f,0)=(1,\new \tau^{(q-1)j}+\new\tau^{(q-1)k}-2c,0)=(1,\new\tau^{(q-1)i},0)$ as $q$ even and $\tau^{(q-1)i}+\tau^{(q-1)j}+\tau^{(q-1)k}=0.$

The case where $T=(1,0,0)$, $N=(0,1,0)$ is similar. In this case the conic $\C$ has equation $\f y^2+\g z^2+xz=0$, for some $\f,\g\in\Fqt$, $\f\neq0$, so for example, $A=( y^2 \f+\g,y,1)$. If $\D$ is  an Andr\'e set with transversal points $(1,0,0),(0,1,0)$, then   we can arrange coordinates to write $\D$ as
$\D=\{(\new\tau^{(q-1)i},1,0)\,|\,i\in  I   \}$ for some  $\new\in\Fqts$. The computations are now similar to the above case.
\end{proof}

\begin{lemma}\label{003} In $\PG(2,q^t)$, $q$ even, $q\geq t$, $t\geq3$ prime,  let $\D$ be  an Andr\'e set with transversal points $(1,0,0)$, $(0,1,0)$. Denote the families of $\Fq$-lines contained in $\D$ by $\F_1,\ldots,\F_{t-1}$ and let $m\in\{1,\ldots,t-1\}$.
Suppose $P_1=(1,\new \tau^{(q-1)i},0)$, $P_2=(1,\new \tau^{(q-1)j},0)$, $P_3=(1,\new \tau^{(q-1)k},0)$ are three distinct points in $\D$ with  $i,j,k\in I $ satisfying
$
\tau^{(q-1)i}+\tau^{(q-1)j}+\tau^{(q-1)k}=0.
$
The $\Fq$-line determined by $P_1,P_2,P_3$ belongs to family $\F_m$ if and only if
$q^t-1$ is a factor of both $(q-1)(q^m-2)(j-i)$ and $(q-1)(q^m-2)(k-i)$.
\end{lemma}

\begin{proof} As in (\ref{e5}),  we can write $\D=
\{(1,\new\tau^{(q-1)i},0)\,|\,i\in  I   \}$ for some  $\nu\in\Fqts$.
Let
$P_1=(1,\new\tau^{(q-1)i},0)$, $P_2=(1,\new\tau^{(q-1)j},0)$, $P_3=(1,\new\tau^{(q-1)k},0)$ be three distinct points in $\D$ with (distinct) $i,j,k\in I$ satisfying
$
\tau^{(q-1)i}+\tau^{(q-1)j}+\tau^{(q-1)k}=0$.
Let $\C$ be a conic with nucleus $N$ and tangent to the line $\li$ at the point $T$ with $\{T,N\}=\{(1,0,0),(0,1,0)\}$.
Then
 by Lemma~\ref{002}, there exists points $A$, $B$, $C$ in $\C$ with $P_1=AB\cap \li$,
$P_2=AC\cap \li$ and  $P_3=BC\cap \li$.
We prove the result for the case where $T=(0,1,0)$ and $N=(1,0,0)$.
The case where  $T=(1,0,0)$, $N=(0,1,0)$ is discussed in the last paragraph of this proof. As $T=(0,1,0)$ and $N=(1,0,0)$, $\C$ has equation $\f x^2+\g z^2+yz=0$ for some $\f,\g\in\Fqt$, $\f\neq0$. So $A=(a,a^2 \f  +\g,1)$, $B=(b, b^2 \f+\g,1)$, $C=(c, c^2 \f+\g,1)$ for some distinct  $a,b,c\in\Fqt$.
We compute $P_1=AB\cap\li=(1, (a+b)\f,0)$, and equate with  $P_1=(1,\new\tau^{(q-1)i},0)$ to get
 \begin{equation}\label{e08}
 a+b=\new\tau^{(q-1)i} \f^{-1}.
 \end{equation}
 Similarly  $a+c= \new\tau^{(q-1)j}\f^{-1}$ ,
$b+c=\new \tau^{(q-1)k}\f^{-1}$.

Before continuing, we introduce some coordinate notation. Section \ref{sec2} introduced the coordinate notation where a point $P=(x,y)\in\PG(1,q^t)$ corresponds to the $(t-1)$-space $[P]$ with  coordinates written as $[P]=\{\langle (sx,sy)\rangle_q\,|\,s\in\Fqt^*\}$. We generalise this  to a coordinate notation for affine points in the Bruck-Bose setting of $\PG(2,q^t)$  in $\PG(2t,q)$. An affine point in $\PG(2,q^t)$ has coordinates that can be written as $Q=(x,y,1)$, the corresponding affine point in $\PG(2t,q)$ has coordinates denoted by $[Q]=(\langle x
\rangle_q,\langle y\rangle_q,1)$.

In $\PG(2t,q)$, $A,B,C$ correspond to  the affine points with coordinates $[A]=(\langle a\rangle_q,\langle a^2\f +\g\rangle_q ,1)$, $[B]=(\langle b\rangle_q,\langle  b^2 \f+\g \rangle_q ,1)$, $[C]=(\langle c\rangle_q,\langle  c^2 \f+\g\rangle_q ,1)$.
The point $Q_1=[A][B]\cap\si$ has coordinates $Q_1=(\langle a\rangle_q+\langle b\rangle_q,\langle  a^2 \f +\g\rangle_q+\langle  b^2 \f+\g \rangle_q,0)=(\langle a+ b\rangle_q,\langle (a+ b)^2 \f\rangle_q,0)$ (as $q$ even). Similarly
$Q_2=[A][C]\cap\si=(\langle a+ c\rangle_q,\langle (a+ c)^2 \f\rangle_q,0)$,
$Q_3=[B][C]\cap\si=(\langle b+ c\rangle_q,\langle (b+ c)^2 \f\rangle_q,0)$.

Let $m\in\{1,\ldots,t-1\}$ and consider the Andr\'e replacement net $\D^m$.
The point $Q_1=(\langle a+ b\rangle_q,\langle (a+ b)^2 \f \rangle_q,0)$ lies in the spread element corresponding to the point $P_1=AB\cap\li$. As $P_1\in\D$, the point $Q_1$ lies in a unique $(t-1)$-space of replacement net $\D^m$. That is, recalling
coordinates from (\ref{e6}), there exists $u_1\in\Fqt^*$ with $ \Nm(u_1)=\Nm(\new)$
such that
$Q_1\in X_{m,u_1}=\{(\langle s\rangle_q,\langle s^{q^m}u_1\rangle_q,0)\,|\, s\in\Fqt^*\}$.
So for some $s\in\Fqts$ and $d\in\Fq^*$, we have
$d(\langle a+ b\rangle_q,\langle (a+ b)^2 \f\rangle_q,0)=(\langle s\rangle_q,\langle s^{q^m}u_1\rangle_q,0)$.
Equating the first coordinate we have
$d\langle a+ b\rangle_q=\langle s\rangle_q$,  and so as $d\in\Fq^*$, we have $s=d(a+b)$.
Equating the  second coordinate gives  $d\langle (a+ b)^2 \f \rangle_q=\langle s^{q^m}u_1\rangle_q$ and so $d(a+ b)^2 \f=s^{q^m}u_1$. Substituting $s=d(a+b)$ and rearranging gives $u_1^{-1}=d^{q^m-1}(a+b)^{q^m-2} \f^{-1}$. As $d\in\Fq^*$, we have $d^{q^m-1} =1$ and so $u_1^{-1}=(a+b)^{q^m-2}\f^{-1}$.
Using (\ref{e08}) gives  $u_1^{-1}=\new^{q^m-2}\tau^{i(q-1)(q^m-2)} \f^{1-q^m}$.

Similarly, $Q_2$  lies in a unique element of $\D^m$, so for some $u_2\in\Fqts$ with  $\Nm(u_2)=\Nm(\new)$, we have $Q_2\in X_{m,u_2}$ and similar calculations give
\begin{equation}\label{e031}
u_2^{-1}=\new^{q^m-2}\tau^{j(q-1)(q^m-2)} \f^{1-q^m}.
\end{equation}
Similarly, $Q_3$  lies in a unique element of $\D^m$, so for some $u_3\in\Fqts$ with  $\Nm(u_3)=\Nm(\new)$, we have   $u_3^{-1}=\new^{q^m-2}\tau^{k(q-1)(q^m-2)}\f^{1-q^m}$.

We now prove the forward direction of the result. Suppose
  the unique  $\Fq$-line containing $P_1,P_2,P_3$ (denoted by $b$)  lies in family $\F_m$. Let $\pi$ be  the unique $\Fq$-plane that is secant to $\li$ and contains $A,B,C$, then $\pi$ meets $\li$ in the $\Fq$-line $b$. In $\PG(2t,q)$, $[\pi]$ is a plane that meets $\si$ in a line $\ell$. Further, $\ell$ meets every element of the $(t-1)$-regulus $[b]$, so  $\ell$ is a ruling line of this $(t-1)$-regulus. As $b$ belongs to family $\F_m$, by Result~\ref{M110}, the line $\ell$ belongs to one $(t-1)$-space of the replacement net $\D^m$, that is, $\ell$ is contained $X_{m,u}$ for some $u\in\Fqts$ with $\Nm(u)=\Nm(\new)$.
Further the points  $[A], [B], [C]$ lie in $[\pi]$, so $Q_1,Q_2,Q_3\in\ell$. That is, $Q_1,Q_2,Q_3$ lie in the same $(t-1)$-space of $\D^m$,  and so $u=u_1=u_2=u_3$.
 As $u_1=u_2$, we have $u_1^{-1}=u_2^{-1}$ and so  
$\new^{q^m-2}\tau^{i(q-1)(q^m-2)}\f^{1-q^m}=\new^{q^m-2}\tau^{j(q-1)(q^m-2)}\f^{1-q^m}$
 (see  (\ref{e031})). Hence
 $\tau^{(j-i)(q-1)(q^m-2)}=1$ and so $q^t-1$ is a factor of $(j-i)(q-1)(q^m-2)$. Similarly $u_1=u_3$ implies $q^t-1$ is a factor of $(k-i)(q-1)(q^m-2)$ as required.

We now prove the converse. Suppose that $q^t-1$ is a factor of both  $(j-i)(q-1)(q^m-2)$ and $(k-i)(q-1)(q^m-2)$.
As $q^t-1$ is a factor of  $(j-i)(q-1)(q^m-2)$, there exists $e\in\mathbb Z$ with $e(q^t-1)=(j-i)(q-1)(q^m-2)$ and so $j(q-1)(q^m-2)=i(q-1)(q^m-2)+e(q^t-1)$. From  (\ref{e031}), we have  
$u_2^{-1}=\new^{q^m-2}\tau^{j(q-1)(q^m-2)} \f^{1-q^m}$ and so
$u_2^{-1}=\new^{q^m-2} \tau^{i(q-1)(q^m-2)\f^{1-q^m} +e(q^t-1)}=\new^{q^m-2}\tau^{i(q-1)(q^m-2)}\f^{1-q^m}=u_1^{-1}$.
 Hence $u_1=u_2$. Similarly we have $u_1=u_3$. That is, $Q_1,Q_2,Q_3$ lie in a common $(t-1)$-space $X_{m,u_1}$ of $\D^m$.
The plane $\alpha$ spanned by $[A], [B],[C]$ meets $\si$ in a line $\ell$ which contains the three points $Q_1,Q_2,Q_3$, so $\ell\subset X_{m,u_1}$.
In $\PG(2,q^t)$, the corresponding set of points is denoted  $\B(\alpha)$, that is
$\B(\alpha)=\{P\in\PG(2,q^t)\,|\, [P]\cap\alpha\neq\emptyset\}$.
So $\B(\alpha)$ is the unique $\Fq$-plane that is secant to $\li$ and contains $A,B,C$.
Further $\B(\alpha)\cap\li$ is the $\Fq$-line denoted $\B(\ell)$ where $\B(\ell)=\{P\in\PG(2,q^t)\,|\, [P]\cap\ell\neq\emptyset\}$.
The points  $P_1=AB\cap \li$,  $P_2=AC\cap \li$,  $P_3=BC\cap \li$ lie in $\B(\alpha)$, so $\B(\ell)$ is the unique $\Fq$-line containing the three points $P_1,P_2,P_3$. Finally, as $\ell$ is contained in the element $X_{m,u_1}$ of $\D^m$, it follows from  Result~\ref{M110} that $\B(\ell)$ belongs to family $\F_m$ as required.

The other case where  $T=(1,0,0)$, $N=(0,1,0)$ is similar.  In this case the conic $\C$ has equation $\f y^2+\g z^2+xz=0$, for some $\f,\g\in\Fqt$, $\f\neq0$, so for example, $A=( y^2 \f +\g,y,1)$ If $\D$ is  an Andr\'e set with transversal points $(1,0,0),(0,1,0)$, then   we can arrange coordinates so
$\D=\{(\new\tau^{(q-1)i},1,0)\,|\,i\in  I   \}$ for some  $\new\in\Fqts$. The replacement sets can be written as
$\D^m=\{ X_{m,u}=\{(\langle s^{q^m}u\rangle_q,\langle s\rangle_q,0)\,|\, s\in\Fqt^*\}\ \,|\,u\in\Fqt^*, \Nm(u)=\Nm(\new)\}.$
The computations are now similar to the above case.
\end{proof}


The next two results are used to show that we can replace the condition that $q^t-1$ is a factor of both $(q-1)(q^m-2)(j-i)$ and $(q-1)(q^m-2)(k-i)$ with the condition that $t\bigm| hm-1$.

\begin{lemma}\label{008}
Suppose $q$ even, $t\geq3$ prime and $m\in\{1,\ldots,t-1\}$, then
$\gcd( q^m-2,\hyp)$ is either $1$ or $2^t-1$. Further, $\gcd( q^m-2,\hyp)=2^t-1$ if and only if
$ t\bigm| hm-1.$
\end{lemma}

\begin{proof}
For $m,n\in\mathbb Z$, it is straightforward to check that
\begin{equation}\label{e007} \gcd(2^m-1,2^n-1)=2^{\gcd( m,n)}-1. \end{equation}
Using this we have
\begin{equation}\label{e007b}
\gcd( 2^{hm-1}-1,2^{h}-1)=2^{\gcd( hm-1,h)}-1=1.
\end{equation}
Hence
   \[
  \begin{split}
   \gcd( q^m-2,\hyp)
  =&\gcd( 2(2^{hm-1}-1),\frac{2^{ht}-1}{2^h-1})\\
  =&\gcd( 2^{hm-1}-1,\frac{2^{ht}-1}{2^h-1})
  \qquad \mbox{ as $\gcd\left( 2,\frac{2^{ht}-1}{2^h-1}\right)=1$}\\
  =&\gcd( 2^{hm-1}-1,2^{ht}-1) 
  \qquad \mbox{  by  (\ref{e007b})}\\
  =&2^{\gcd( hm-1,ht)}-1\qquad\qquad \mbox{by  (\ref{e007})}\\
  =&2^{\gcd( hm-1,t)}-1\qquad \qquad\mbox{ as $\gcd(  hm-1,h)=1$.}\\
  \end{split}
  \]
As $t$ is a prime, either $\gcd( hm-1,t)=1$ or $\gcd( hm-1,t)=t$. So $\gcd( q^m-2,\hyp)$ is 1 or $2^t-1$, the latter case occurs if and only if $t\bigm|hm-1$.
\end{proof}

\begin{lemma}\label{011}
Suppose $q$ even, $t\geq3$ prime and $m\in\{1,\ldots,t-1\}$.
\begin{enumerate}
\item If $q^t-1$ is a factor of $(q-1)(q^m-2)(j-i)$ for some distinct $i,j\in I $, then  $ t\bigm| hm-1$.
\item   If  $ t\bigm| hm-1$, then   there exists
   distinct $i,j,k\in I $ satisfying $\tau^{(q-1)i}+\tau^{(q-1)j}+\tau^{(q-1)k}=0$ and $q^t-1$ is a factor of both $(q-1)(q^m-2)(j-i)$ and $(q-1)(q^m-2)(k-i)$.
   \end{enumerate}
\end{lemma}

\begin{proof} To prove part 1, suppose $q^t-1$ is a factor of $(q-1)(q^m-2)(j-i)$ for some $i,j\in I $. To show that $ t\bigm| hm-1$, it follows from by Lemma \ref{008} that we only need show that $\gcd( q^m-2,\hyp) \neq 1$. Suppose $\gcd( q^m-2,\hyp)= 1$. As $q^t-1$ is a factor of $(q-1)(q^m-2)(j-i)$, it follows that    $\hyp$ is a factor of $j-i$, which contradicts  $0\le i,j<\hyp$.

To prove part 2, suppose that $ t\bigm| hm-1$. As
 $\tau$ is a multiplicative generator for  $\Fqt^*$,  the subfield  $\mathbb F_{2^t}^*$ has multiplicative generator $\omega=\tau^{\frac{q^t-1}{2^t-1}}$.
 Now $1+\omega\in \mathbb F_{2^t}^*$, so there exists unique
 $a\in \{2,3,\ldots,2^t-2\}$ with $\omega^a=1+\omega$.
As $ t\bigm| hm-1$, it follows from  Lemma~\ref{008} that
 $2^t-1\bigm| \hyp$, so setting
\begin{equation}\label{e09}
  i=0,\quad j=\frac 1{2^t-1}\left(\hyp\right),\quad k=\frac a{2^t-1}\left(\hyp\right),
  \end{equation}
  gives distinct $i,j,k\in I $.
Using (\ref{e09}) we have
  \begin{eqnarray*}
  \tau^{(q-1)i}+\tau^{(q-1)j}+\tau^{(q-1)k}
    &=& 1+\tau^{(q-1)\frac 1{2^t-1}\hyp}+\tau^{(q-1)\frac a{2^t-1}\hyp}\\
    &=&\omega^0+\omega^1+\omega^a.
        \end{eqnarray*}
This equals $0$ by the definition of $a$.
Further, as $t\bigm| hm-1$, by Lemma~\ref{008} we have
$2^t-1\bigm| q^m-2$. By (\ref{e09}), $i,j,k$ are all multiples of $\frac 1{2^t-1}\left(\hyp\right)$, so it follows that both $(q-1)(q^m-2)(j-i)$ and $(q-1)(q^m-2)(k-i)$
are divisible by
  \[
  (q-1)\times (2^t-1)\times \frac 1{2^t-1}\left(\hyp\right)=q^t-1,
  \]
  as required.
\end{proof}

\subsubsection{Tangent conics - inherited arcs}

We now use these technical lemmas to prove the existence of inherited arcs
 in the Andr\'e plane.

\begin{lemma}\label{012}
In $\PG(2,q^t)$, $q$ even, $q\geq t$, $t\geq3$ prime,  let $\C$ be a conic with nucleus $N$ and tangent to the line $\li$ at the point $T$ with $\{T,N\}=\{(1,0,0),(0,1,0)\}$ and let
$\D$ be  an Andr\'e set with transversal points $T,N$ (so coordinates for $\D$ and $\D^m$ are given in equations \textup{(\ref{e5})} and \textup{(\ref{e6})}).
For each $m\in\{1,\ldots,t-1\}$, the set $\D_m(\C)$ is an arc in the Andr\'e plane $\P(\D^m)$ if and only if
$t$ does not divide $h m -1$. 
\end{lemma}

\begin{proof} We prove the contrapositive, namely we show that $\D_m(\C)$ is not an arc of $\P(\D^m)$ if and only if $t\bigm| hm-1$.
Suppose $\D_m(\C)$ is not an arc of $\P(\D^m)$. Then there are three affine points $A,B,C\in\C$ with $A,B,C$ collinear in $\P(\D^m)$, so $A,B,C$ lie  in a common $\D^m$-block. It follows from Result~\ref{001} that  the $\Fq$-line determined by the three points $P_1=AB\cap\li$, $P_2=AC\cap\li$, $P_3=BC\cap\li$ belongs to family $\F_m$. In particular this means that $P_1,P_2,P_3\in\D$, and so $P_1=(1,\tau^{(q-1)i},0)$,
$P_2=(1,\tau^{(q-1)j},0)$,
$P_3=(1,\tau^{(q-1)k},0)$ for distinct $i,j,k\in I $.
By Lemma~\ref{002}, $i,j,k$ satisfy
$\tau^{(q-1)i}+\tau^{(q-1)j}+\tau^{(q-1)k}=0$. It follows from Lemma~\ref{003} that $q^t-1$ is a factor of both $(q-1)(q^m-2)(j-i)$ and $(q-1)(q^m-2)(k-i)$. Hence by Lemma~\ref{011}, we have $t\bigm| hm-1$.

Conversely suppose $t\bigm| hm-1$. By Lemma~\ref{011}, there exists distinct $i,j,k\in I $ satisfying
$\tau^{(q-1)i}+\tau^{(q-1)j}+\tau^{(q-1)k}=0$ and $q^t-1$ is a factor of both $(q-1)(q^m-2)(j-i)$ and $(q-1)(q^m-2)(k-i)$.
By Lemma~\ref{003}, the
$\Fq$-line determined by the three points $P_1=(1,\tau^{(q-1)i},0)$,
$P_2=(1,\tau^{(q-1)j},0)$,
$P_3=(1,\tau^{(q-1)k},0)$ belongs to family $\F_m$.
By Lemma~\ref{002}, there exists three points  $A,B,C\in\C$ with $AB\cap\li=P_1$, $AC\cap\li=P_2$, $BC\cap\li=P_3$. By Result~\ref{001}, the three points $A,B,C$ lie in a common $\D^m$-block. Hence $A,B,C$ are collinear in $\P(\D^m)$ and so
$\D_m(\C)$ is not an arc.
\end{proof}

\begin{theorem}\label{013}
In $\PG(2,q^t)$, $q$ even, $q\geq t$, $t\geq3$ prime,  let $\C$ be a conic with nucleus $N$ and tangent to the line $\li$ at the point $T$ with $\{T,N\}=\{(1,0,0),(0,1,0)\}$ and let
$\D=\{(1,\new\tau^{(q-1)i})\,|\,i\in  I   \}$   (for some $\new\in\Fqt^*$) be  an Andr\'e set with transversal points $T,N$. Perform an Andr\'e replacement using the Andr\'e replacement net $\D^m=\{ X_{m,u}=\{\langle(s,s^{q^m}u)\rangle_q\,|\, s\in\Fqt^*\}\ \,|\,u\in\Fqt^*, \Nm(u)=\Nm(\new)\}$; and let $\D_m(\C)$ denote the set of points in $\P(\D^m)$ corresponding to the affine points of $\C$.
\begin{enumerate}
\item If $h\equiv 0\pmod{t}$, then for each $m\in\{1,\ldots,t-1\}$, $\D_m(\C)$ is an arc which can be completed to an oval in the Andr\'e plane $\P(\D^m)$.
\item If $h\not\equiv 0\pmod{t}$, then  there is a unique $n\in\{1,\ldots,t-1\}$ for which $t$  divides $h n -1 $.  For this $n$  the following holds:
\begin{enumerate}
\item
$\D_n(\C)$ is not an arc in the Andr\'e plane $\P(\D^n)$;
\item for each $m\in\{1,\ldots,t-1\}\setminus\{n\}$,
$\D_m(\C)$ is an arc  which can be completed to an oval in the Andr\'e plane $\P(\D^m)$.
\end{enumerate}
\end{enumerate}
\end{theorem}

\begin{proof}
 Suppose $h\equiv 0\pmod{t}$, that is,
$t\bigm| h$. Then $t$ does not divide $hm-1$ for any $m\in\{1,\ldots,t-1\}$. So by Lemma~\ref{012}, $\D_m(\C)$ is an arc in the Andr\'e plane $\P(\D^m)$ for each $m\in\{1,\ldots,t-1\}$. As the point $T$ does not lie in $\D$, all the lines of $\P(\D^m)$ through $T$ are also lines of $\PG(2,q^t)$, so $T$  does not lie on a $3$-secant of $\D_m(\C)$. Hence the set of points $\D_m(\C)\cup T$ form an oval of $\P(\D^m)$.

 Suppose $h\not\equiv 0\pmod{t}$,
 then $\gcd( h,t)=1$ as $t$ is a prime.
Hence for $m\in\{1,\ldots,t-1\}$, we have $\gcd(h m,t)=1$, and so $h m\not\equiv0\pmod{t}$.
Next suppose there exists distinct $m_1,m_2\in\{1,\ldots,t-1\}$ with  $h m_1\equiv h m_2 \pmod t$. Then
 $ t \bigm|h m_1-h m_2 $, (since
$\gcd( t,h)=1$). 
As $0<m_1,m_2<t$, it follows that $m_1=m_2$. That is, we have shown that the set
 $\{h m\pmod t\,|\,m=1,\ldots,t-1\}$ equals the set $\{1,\ldots,t-1\}$. Hence there is a unique $n\in\{1,\ldots,t-1\}$ with $h n\equiv 1\pmod{t}$,  that is,  there is a unique $n\in\{1,\ldots,t-1\}$ for which $t$  divides $h n-1$.

It follows from  Lemma~\ref{012} that $\D_m(\C)$ is an arc in the Andr\'e plane $\P(\D^m)$ for each $m\in\{1,\ldots,t-1\}\setminus\{n\}$, and $\D_n(\C)$ is not an arc in $\P(\D^n)$. Moreover we can complete the arcs to ovals by adding the point $T$.
\end{proof}

As all Andr\'e sets are projectively equivalent, we have the following general result.

\begin{theorem}\label{014}
In $\PG(2,q^t)$, $q$ even, $q\geq t$, $t\geq3$ prime,  let $\C$ be a conic with nucleus $N$ and tangent to the line $\li$ at the point $T$ and let
$\D$ be  an Andr\'e set with transversal points $T,N$.  Perform an Andr\'e replacement using the Andr\'e replacement net $\D^m$; and let $\D_m(\C)$ denote the set of points in $\P(\D^m)$ corresponding to the affine points of $\C$.
\begin{enumerate}
\item If $h\equiv 0\pmod{t}$, then for each $m\in\{1,\ldots,t-1\}$, $\D_m(\C)$ is an arc which can be completed to an oval in the Andr\'e plane $\P(\D^m)$.
\item If $h\not\equiv 0\pmod{t}$, then  there is a unique $n\in\{1,\ldots,t-1\}$ for which $\D_n(\C)$ is not an arc in the Andr\'e plane $\P(\D^n)$; for  the remaining $t-2$ values $m\in\{1,\ldots,t-1\}\setminus\{n\}$,
$\D_m(\C)$ is an arc  which can be completed to an oval in the Andr\'e plane $\P(\D^m)$.
\end{enumerate}
\end{theorem}

\section{Conclusion}\label{sec-con}

Theorem~\ref{thm000} showed that a conic of $\PG(2,q^t)$, $q$ even, which is secant or exterior to $\li$ never inherits to an arc of an Andr\'e plane. The tangent conic case is more complex. If a conic $\C$ of $\PG(2,q^t)$, $q$ even, is tangent to $\li$ at point $T$ with nucleus $N$, there are a number of cases to consider. Theorem~\ref{201}
shows that  when $T,N\in\D$, $\C$ does not inherit to an arc. Theorem~\ref{014} looked at a case where $\C$ does inherit to an arc (namely when $T,N$ are the transversal points  of the Andr\'e set $\D$).

The remaining cases to consider include: the cases where exactly one of $T,N$ lies in $\D$, and the other either is, or is not, a transversal point of $\D$; and  the case where $T,N\notin\D$, with either zero or one of $T,N$ being a transversal point of $\D$.
Computer searches when $t=3$ for small $q$ using Magma~\cite{magma} tested all of these cases and did not find any examples where the conic inherited to an arc.

In particular,  computer results suggest that  the inherited arc described in Result~\ref{120e}(3b) does not generalise to the case $t= 3$.
It seems likely  that there are no other cases which lead to inherited arcs and the following conjecture encompasses this in the case $t=3$, that is, in $\PG(2,q^3)$.

\begin{conjecture}
In $\PG(2,q^3)$, $q$ even, $q>2$, let $\C$ be a conic with nucleus $N$ and tangent to the line $\li$ at the point $T$ and let
$\D$ be  an Andr\'e set of $\li$ with transversal points $E,F$.
If $|\{T,N\}\cap\{E,F\}|<2$, then
$\D_m(\C)$ is not an arc in the Andr\'e plane $\P(\D^m)$ for any $m\in\{1,2\}$.
\end{conjecture}


Another interesting question arises from a conjecture in \cite[p556]{GlynnStein} that considers the arc constructed in Result~\ref{120e}(3c); they conjecture that all ovals through two conjugate points in the Hall plane of order $q=2^{2k}$ derive from conics in the Desarguesian plane. An interesting generalisation
is to an Andr\'e plane $\P(\D^m)$ of order $q^t$, $q$ even, $t\geq3$ prime that is obtained from the Desarguesian plane $\PG(2,q^t)$ using one Andr\'e plane net replacement. We can ask whether an oval in $\P(\D^m)$ that contains  the two transversal points on $\li$ is inherited from a conic of $\PG(2,q^t)$.

\end{document}